\documentclass[12pt]{amsart}
\usepackage{amsmath, amssymb}
\usepackage{hyperref}
\usepackage{bbm}
\DeclareMathOperator{\supp}{supp}

\newtheorem{theorem}{Theorem}[section]
\newtheorem{lemma}[theorem]{Lemma}
\newtheorem{corollary}[theorem]{Corollary}

\theoremstyle{definition}
\newtheorem{definition}[theorem]{Definition}

\theoremstyle{remark}
\newtheorem{remark}[theorem]{Remark}

\numberwithin{equation}{section}

\begin{document}

\title[Packing dimension of H\"{o}lder images]{A threshold for full packing dimension of H\"{o}lder images of sets and measures}


%    author one information
\author{Nicolas Angelini}

\curraddr{}
\address{Department of Mathematics, Università degli Studi di Trento. Trento, Italy.}
\address{Instituto de Matemática Aplicada San Luis (IMASL), CONICET, Argentina.  
Departamento de Matemática, FCFMyN, Universidad Nacional de San Luis, Argentina.}
\email{nicolas.angelini.2015@gmail.com}
\thanks{}

%    \subjclass is required.
\subjclass[2010]{Primary }

\date{\today}

\begin{abstract}
We study when the image of a measure under a random H\"{o}lder map attains
full packing dimension. Our main tool is a family of packing intermediate
dimension profiles $\dim_{P,\theta}^{s}\mu$, indexed by $\theta\in(0,1]$
and $s>0$, which refine the packing dimension profiles of Falconer and
Howroyd and reduce to
them at $\theta=1$.
For a large family of random $\alpha$-H\"{o}lder maps
$f_\omega:\mathbb{R}^n\to\mathbb{R}^m$, which includes index-$\alpha$
fractional Brownian motion as a particular case, we prove that for every
compactly supported Borel probability measure $\mu$ on $\mathbb{R}^n$,
$$\dim_P \mu_{f_\omega}=m \quad\text{almost surely}
\quad\Longleftrightarrow\quad \alpha m\leq \lim_{\theta\to 0}\dim_{P,\theta}^n \mu.$$
We further study the profiles $\dim_P^s\mu$ themselves, obtaining a
quantitative lower bound and a
Marstrand-type identity for their limiting behavior as $\theta\to 0$.
Finally, we obtain the analogous
characterization for analytic sets: 
$$\dim_P f_\omega(E)=m \quad\text{almost surely}
\quad\Longleftrightarrow\quad \alpha m\leq \lim_{\theta\to 0}\dim_{P,\theta}^n E,$$
where $\dim_{P,\theta}^n E=\sup\{\dim_{P,\theta}^n \mu :\, \mu\in \mathcal{M}_c^+(E) \}$.
\end{abstract}


\maketitle

\section{Introduction}

Packing dimension $\dim_P$ was introduced as a dual concept to Hausdorff
dimension and has since been studied extensively, both for sets and
measures; see for example \cite{falconer2013fractal, mattila1999geometry}.

A fundamental problem in fractal geometry is understanding how dimension
behaves under mappings: given a measure $\mu$ and a map $f$, how does
$\dim \mu_f$ relate to $\dim \mu$, where $\dim$ denotes any notion of
dimension and $\mu_f$ denotes the image of $\mu$ under $f$?
A natural class of examples is provided by $\alpha$-H\"{o}lder maps,
which includes orthogonal projections (the case $\alpha = 1$) and
fractional Brownian motion.

For index-$\alpha$ fractional Brownian motion $X:\mathbb{R}^n\to\mathbb{R}^m$
and a compactly supported finite Borel measure $\mu$ on $\mathbb{R}^n$,
Kahane \cite{kahane1985some} proved the set version of the formula
\begin{equation}\label{eq:kahane}
    \dim_H \mu_X = \min\!\left\{m,\, \frac{\dim_H \mu}{\alpha}\right\}
    \quad \text{a.s.},
\end{equation}
from which the measure version follows by applying it to a Borel set of
full $\mu$-measure, where $\mu_X$ is the image of $\mu$ under $X$; see
\cite{mattila1999geometry} for background on Hausdorff dimension.
Unfortunately, packing dimension does not admit analogous formulas in
general. Xiao \cite[Theorem 3.1]{xiao1997packing} proved that for any
finite Borel measure $\mu$ on $\mathbb{R}^n$,
\begin{equation}\label{hausineq}
  \dim_P \mu_X = \frac{1}{\alpha}\,\dim_P^{\alpha m}\mu
  \quad \text{a.s.},
\end{equation}
where $\dim_P^{s}\mu$ denotes the $s$-dimensional packing dimension
profile introduced by Falconer and Howroyd \cite{falconer1997packing};
see Definition~\ref{def:packing_profile} and Theorem~\ref{ProjPacking}.
This gives a complete formula for the packing dimension of the image,
expressed in terms of a profile that encodes multiscale geometric
information about $\mu$. However, such profiles are difficult to compute
in practice, and a natural question is: under which conditions can one
recover formulas analogous to~\eqref{eq:kahane}?

In \cite{Angelini_Packing} a threshold parameter $D(\mu)$ was
introduced, defined by
\[
  D(\mu) := \lim_{\theta \to 0}\,\dim_{P,\theta}\mu,
\]
where $\{\dim_{P,\theta}\mu\}_{\theta\in(0,1]}$ is a spectrum of
intermediate dimensions indexed by $\theta$, which at $\theta=1$ recovers
the packing dimension of $\mu$; see
Definition~\ref{def:packing_intermediate}. This parameter characterizes
the threshold at which the $m$-dimensional profile attains its maximum:
\[
  \dim_P^m\mu = m \;\;\Longleftrightarrow\;\; m \leq D(\mu).
\]
However, the proof of this characterization relies essentially on the
fact that $\dim_P^m\mu$ equals the almost sure packing dimension of the
image of $\mu$ under orthogonal projections onto $m$-dimensional
subspaces of $\mathbb{R}^n$, and therefore requires the profile exponent
to be a positive integer. As a consequence, the result cannot be directly
applied to the almost sure packing dimension of images under fractional
Brownian motion, where the natural exponent $\alpha m$ is in general not
an integer.

In this article we fill this gap. Our first main result concerns a family
of intermediate packing dimension profiles $\dim_{P,\theta}^s\mu$,
defined for every real exponent $s>0$, which extends the intermediate
dimension of \cite{Angelini_Packing} and interpolates Xiao's profile at
$\theta=1$. We show that, for every $s\in(0,n]$ and $\theta\in(0,1]$,
\begin{equation}\label{eq:packing_inequality}
  (1-\theta)\min\{s,\dim_{P,\theta}\mu\}
  \;\leq\; \dim_P^s\mu \;\leq\; s,
\end{equation}
and we determine the limiting behaviour of the profiles as $\theta\to 0$,
obtaining an analogue of Marstrand's theorem:
\[
  \lim_{\theta\to 0}\,\dim_{P,\theta}^s\mu
  \;=\;
  \min\!\left\{s,\,D(\mu)\right\};
\]
see Theorem~\ref{Lemma:cotainferior}. For the intermediate dimensions of measures interpolating between the
Hausdorff and Minkowski dimensions, introduced in
\cite{ANGELINI2026130039}, no such identity was obtained, the
corresponding estimates there hold only as inequalities
(\cite[Theorem 6.2]{ANGELINI2026130039}), and requires $\dim_A\mu<\infty$,
that is, $\mu$ doubling (\cite[Lemma 4.1.1]{fraser2020assouad}). In the packing setting no such restriction is
needed, and we obtain the identity for every compactly supported measure.

Letting $\theta\to 0$ in
\eqref{eq:packing_inequality} shows in particular that $\dim_P^s\mu = s$
whenever $s\leq D(\mu)$. Moreover, in the extremal case
$D(\mu)=\dim_P\mu$, which holds, for instance, whenever
$\dim_P\mu=\dim_A(\operatorname{supp}\mu)$, we obtain
$\dim_P^s\mu=\min\{s,\dim_P\mu\}$, and Xiao's formula \eqref{hausineq}
reduces to the packing analogue of \eqref{eq:kahane}, namely
$\dim_P\mu_X=\min\{m,\alpha^{-1}\dim_P\mu\}$ almost surely; see
Corollary~\ref{cor:full_dim_fbm}.

Our second main result removes the integrality restriction described
above. We show that $D(\mu)$ characterizes the exact phase transition for
full-dimensional images for a large family of random $\alpha$-H\"{o}lder
maps $f_\omega:\mathbb{R}^n\to\mathbb{R}^m$, which includes fractional
Brownian motion as a particular case:
\[
  \dim_P\mu_{f_\omega} = m \quad\text{a.s.}
  \;\;\Longleftrightarrow\;\;
  \alpha m \leq D(\mu);
\]
see Theorem~\ref{thm:main}. The proof proceeds by computing
$\dim_{P,\theta}\mu_{f_\omega}$ explicitly for this family (see
Lemmas~\ref{Th1} and~\ref{Th2}), extending the capacity-theoretic ideas
of \cite{Burrell} from the intermediate dimension setting of sets to the
packing dimension setting of measures, and passing from the
integer-parameter profiles for orthogonal projections developed in
\cite{Angelini_Packing} to the real-parameter profiles required for
general H\"{o}lder maps. In particular, at $\theta=1$ our computation
recovers Xiao's formula \eqref{hausineq} as a special case.

Finally, we obtain the analogue of our main result for sets. Following
\cite{xiao1997packing}, the packing dimension of the image of an analytic
set can be recovered from the measures it supports, which allows us to
transfer the threshold phenomenon from measures to sets. Defining
\[
  D(E):=\lim_{\theta\to 0}\,\sup\{\dim_{P,\theta}\mu\,:\,
  \mu\in\mathcal{M}_c^+(E)\},
\]
where $\mathcal{M}_c^+(E)$ denotes the family of finite Borel measures
with compact support contained in $E$, we prove that for an analytic set
$E\subset\mathbb{R}^n$
\[
  \dim_P f_\omega (E) = m \quad\text{a.s.}
  \;\;\Longleftrightarrow\;\;
  \alpha m \leq D(E);
\]
for a large family of random $\alpha$-H\"{o}lder
maps $f_\omega:\mathbb{R}^n\to\mathbb{R}^m$, which includes fractional
Brownian motion as a particular case,
see Theorem~\ref{thm:main_sets} and Corollary~\ref{cor:final}.

The paper is organized as follows. In Section~2 we recall the necessary
background on packing dimension, dimension profiles, fractional Brownian
motion, and the intermediate dimensions and threshold parameter of
\cite{Angelini_Packing}. In Section~3 we introduce the intermediate
packing dimension profiles $\dim_{P,\theta}^{s}\mu$ and state our two
main theorems, the estimates for these profiles
(Theorem~\ref{Lemma:cotainferior}) and the threshold characterization of
full-dimensional images under random H\"{o}lder maps
(Theorem~\ref{thm:main}); we also establish there the deterministic and
probabilistic bounds on image measures (Lemmas~\ref{Th1} and~\ref{Th2})
and derive the intermediate profiles of images under fractional Brownian
motion (Corollary~\ref{cor:FBM}). Section~\ref{sec:proofs} contains the
proofs of both main theorems, together with the extremal case in which
the packing analogue of Kahane's formula is recovered
(Corollaries~\ref{cor:full_dim_profile} and~\ref{cor:full_dim_fbm}).
Finally, in Section~5 we transfer the threshold to analytic sets
(Theorem~\ref{thm:main_sets}) and deduce the corresponding statement for
images under random H\"{o}lder maps (Corollary~\ref{cor:final}).

\section{Preliminaries}\label{sec:prelim}

We write $B(x,r)$ for the closed ball of centre $x$ and radius $r$, and
for a bounded set $E\subset\mathbb{R}^n$ we write $|E|$ for its diameter.
Throughout the paper, $\mu$ denotes a Borel probability measure with
compact support on $\mathbb{R}^n$, that is $\mu(\mathbb{R}^n)=1$, and we
write $\supp\mu$ for its support. Given a Borel measurable function
$f:\mathbb{R}^n\to\mathbb{R}^d$, the image (or pushforward) measure of
$\mu$ under $f$ is
$$\mu_f(B):=\mu\bigl(f^{-1}(B)\bigr)=\mu\{x\in\mathbb{R}^n:f(x)\in B\},
\qquad B\subset\mathbb{R}^d \text{ Borel}.$$
Given $0<\alpha\leq 1$, a function $f:A\subseteq\mathbb{R}^n\to\mathbb{R}^m$
is \emph{$\alpha$-H\"{o}lder continuous} if there exists a constant $c>0$
such that
$$|f(x)-f(y)|\leq c\,|x-y|^{\alpha}\qquad\text{for all }x,y\in A.$$
For a nonempty set $E\subset\mathbb{R}^n$, the \emph{Assouad dimension}
$\dim_A E$ is the infimum of the exponents $s\geq 0$ for which there
exists $C>0$ such that, for all $0<r<R$ and all $x\in E$, the set
$E\cap B(x,R)$ can be covered by at most $C(R/r)^{s}$ balls of radius
$r$; see \cite{fraser2020assouad} for its properties. When
$E=\supp\mu$ is the compact support of a Borel probability measure, one
has
$$\dim_P\mu\leq\dim_P E\leq\dim_A E\leq n.$$

\subsection*{Packing dimension of measures and dimension profiles}

The packing dimension of a Borel measure $\mu$ is defined by
\[
\dim_P\mu=\inf\{\dim_P E : E \text{ is a Borel set with }\mu(E)>0\},
\]
where $\dim_P E$ denotes the packing dimension of the set $E$. By
\cite[Lemma 4.1]{hu1994fractal}, it admits the equivalent
characterization
\begin{equation}\label{eq:hu}
\dim_P\mu=\sup\left\{t\geq 0 : \liminf_{r\to 0}\frac{\mu(B(x,r))}{r^{t}}=0
\ \text{ for }\mu\text{-a.e.\ }x\in\mathbb{R}^n\right\}.
\end{equation}

Falconer and Howroyd \cite{falconer1997packing} introduced the packing
dimension profiles, which recover the packing dimension of the typical
projection of a measure onto a lower-dimensional subspace.

\begin{definition}\label{def:packing_profile}
For a finite Borel measure $\mu$ on $\mathbb{R}^n$ and $0\leq m\leq n$,
set
$$F_m^\mu(x,r)=\int_{\mathbb{R}^n}\min\{1,\,r^m|x-y|^{-m}\}\,d\mu(y),$$
and define the \emph{$m$-dimensional packing dimension profile} of $\mu$
by
$$\dim_P^m\mu=\sup\left\{t\geq 0 : \liminf_{r\to 0}\frac{F_m^{\mu}(x,r)}{r^{t}}=0
\ \text{ for }\mu\text{-a.e.\ }x\in\mathbb{R}^n\right\}.$$
\end{definition}

At the ambient exponent this reduces to the packing dimension: by
\cite[Corollary 3]{falconer1997packing},
\begin{equation}\label{Corollary3}
\dim_P^n\mu=\dim_P\mu.
\end{equation}
Xiao \cite{xiao1997packing} showed that these profiles, at the possibly
non-integer exponent $\alpha m$, give the almost sure packing dimension
of the image of a measure under fractional Brownian motion.

\begin{theorem}[{\cite[Theorem 3.1]{xiao1997packing}}]\label{ProjPacking}
Let $\mu$ be a finite Borel measure on $\mathbb{R}^n$ and let
$X:\mathbb{R}^n\to\mathbb{R}^m$ be an index-$\alpha$ fractional Brownian
motion. Then, with probability $1$,
$$\dim_P\mu_X=\frac{1}{\alpha}\,\dim_P^{\alpha m}\mu.$$
\end{theorem}

For an analytic set $E\subset\mathbb{R}^n$, let $\mathcal{M}_c^+(E)$
denote the family of finite Borel measures with compact support
contained in $E$. It is well known (see \cite{hu1994fractal}, or
\cite[(4.1)]{falconer1997packing}) that
\begin{equation}
\dim_P E=\sup\{\dim_P\mu : \mu\in\mathcal{M}_c^+(E)\},
\end{equation}
which led Falconer and Howroyd to define the $m$-dimensional packing
dimension profile of $E$ by
$\dim_P^m E:=\sup\{\dim_P^m\mu : \mu\in\mathcal{M}_c^+(E)\}$, see
\cite[(5.3)]{falconer1997packing}.

\subsection*{Fractional Brownian motion}

Given $0<\alpha<1$ and $m,n\in\mathbb{N}$, an \emph{index-$\alpha$
fractional Brownian motion} is a random field
$X:\mathbb{R}^n\to\mathbb{R}^m$ of the form $X=(X_1,\ldots,X_m)$, where
$X_1,\ldots,X_m$ are independent real-valued random fields
$X_i:\mathbb{R}^n\to\mathbb{R}$ satisfying:
\begin{enumerate}
    \item[i)] $X_i(0)=0$;
    \item[ii)] $X_i$ is continuous with probability $1$;
    \item[iii)] the increments $X_i(x)-X_i(y)$ are normally distributed
    with mean $0$ and variance $|x-y|^{2\alpha}$, for all
    $x,y\in\mathbb{R}^n$.
\end{enumerate}
It is well known that, almost surely, the sample paths of $X$ are
$(\alpha-\varepsilon)$-H\"{o}lder continuous on compact sets for every
$\varepsilon>0$. We refer the reader to the classical text of Kahane
\cite{kahane1985some} for a detailed account of fractional Brownian
motion and related stochastic processes.

\subsection*{Intermediate dimension and threshold parameter}

Recently, in \cite{Angelini_Packing}, the packing intermediate dimension
$\dim_{P,\theta}\mu$ and the threshold parameter $D(\mu)$ of a Borel
probability measure were introduced.

\begin{definition}\label{def:packing_intermediate}
Let $\mu$ be a Borel probability measure with compact support on
$\mathbb{R}^n$ and let $\theta\in(0,1]$. For $0\leq t\leq s$ set
$$\phi_{r,\theta}^{t,s}(z)=\min\{1,\,r^t|z|^{-t},\,
r^{\theta(s-t)+t}|z|^{-s}\},$$
 
$$F_{t,s,\theta}^{\mu}(x,r)=\int\phi_{r,\theta}^{t,s}(x-y)\,d\mu(y),$$
and define the \emph{$\theta$-packing intermediate dimension} of $\mu$ by
$$\dim_{P,\theta}\mu=\sup\left\{t\in[0,n] : \liminf_{r\to 0}
r^{-t}F_{t,n,\theta}^{\mu}(x,r)=0 \ \text{ for }\mu\text{-a.e.\ }
x\in\mathbb{R}^{n}\right\}.$$
\end{definition}

\begin{definition}\label{maindefinition}
The \emph{threshold parameter} of a Borel probability measure $\mu$ with
compact support on $\mathbb{R}^n$ is
$$D(\mu):=\lim_{\theta\to 0}\dim_{P,\theta}\mu.$$
\end{definition}

We shall repeatedly use the following equivalent description of
$\dim_{P,\theta}\mu$, established in \cite{Angelini_Packing}, in terms of
the measure of balls along windows of scales.

\begin{theorem}[{\cite[Theorem 4.6]{Angelini_Packing}}]\label{newdef}
Let $\mu$ be a Borel probability measure with compact support on
$\mathbb{R}^n$ and let $\theta\in(0,1]$. Then
\begin{equation}\label{eq:newdef-formula}
\dim_{P,\theta}\mu=\sup\left\{
\begin{array}{c}
s\geq 0 \, : \, \forall\, \varepsilon>0 \text{ and } \delta_0>0\
\exists\, \delta(x)<\delta_0 \text{ such that}\\[2mm]
r^{-s}\mu(B(x,r))<\varepsilon\ \ \forall\, r\in[\delta(x)^{1/\theta},\,\delta(x)],
\ \ \text{ for }\mu\text{-a.e.\ }x
\end{array}
\right\}.
\end{equation}
\end{theorem}

Finally, we record two estimates from \cite{Angelini_Packing} that will
be used in Section~\ref{sec:proofs}.

\begin{lemma}[{\cite[Lemma 3.2]{Angelini_Packing}}]\label{lem:assouad_measure_growth}
Let $\mu$ be a Borel probability measure on $\mathbb{R}^n$ with compact
support $E:=\supp\mu$, and let $a\in(0,1)$. Fix $\varepsilon>0$ and set
$s:=\dim_A E$. Then for $\mu$-almost every $x\in E$ there exists
$\rho_0(x)>0$ such that for all $0<\rho<\rho_0(x)$ and all $r$ with
$\rho^{a}\leq r\leq 1$,
\begin{equation}\label{eq:assouad_growth}
\mu\bigl(B(x,r)\bigr)\leq
\left(\frac{4r}{\rho}\right)^{s(1+\varepsilon)}\mu\bigl(B(x,\rho)\bigr).
\end{equation}
\end{lemma}

\begin{lemma}[{\cite[Lemma 4.7]{Angelini_Packing}}]\label{lem:threshold_lower_bound}
Let $\mu$ be a Borel probability measure on $\mathbb{R}^n$ with compact
support $E:=\supp\mu$ and let $\theta\in(0,1]$. Then
$$\dim_A E-\frac{\dim_A E-\dim_P\mu}{\theta}\leq\dim_{P,\theta}\mu.$$
\end{lemma}

In particular, if $\dim_P\mu=n$ then $\dim_{P,\theta}\mu=n$ for every
$\theta\in(0,1]$, and hence $D(\mu)=n$.

\section{Intermediate packing dimension profiles and images under H\"{o}lder maps}

In this section we introduce the intermediate packing dimension profiles,
extending the intermediate dimension $\dim_{P,\theta}\mu$ of
\cite{Angelini_Packing} to arbitrary exponents, and we state our two main
results: Theorem~\ref{Lemma:cotainferior}, which describes the limiting
behavior of the profiles as $\theta\to 0$, and Theorem~\ref{thm:main},
the threshold characterization of full packing dimension for images under
random H\"{o}lder maps; their proofs are given in
Section~\ref{sec:proofs}.

\begin{definition}\label{def:intermediate_profile}
Let $\mu$ be a Borel probability measure with compact support on
$\mathbb{R}^n$, $\theta\in(0,1]$ and let $s>0$. We define the intermediate packing dimension profile of $\mu$ by
$$\dim_{P,\theta}^s\mu = \sup\left\{t\in[0,s]\, :\, \liminf_{r\to 0}
r^{-t}F_{t,s,\theta}^{\mu}(x,r)=0 \ \text{ for } \mu\text{-a.e.\ }
x\in \mathbb{R}^{n}\right\},$$
where $F_{t,s,\theta}^{\mu}(x,r)$ is defined as in Definition \ref{def:packing_intermediate}.
\end{definition}

\begin{remark}\label{thm:lipschitz_profile}
     For $0\leq t\leq s_1\leq s_2$ we have
$\phi_{r,\theta}^{t,s_2}(z)\leq\phi_{r,\theta}^{t,s_1}(z)$, since the two
kernels differ only in their last term and
$r^{\theta(s_2-t)+t}|z|^{-s_2}\leq r^{\theta(s_1-t)+t}|z|^{-s_1}$ whenever
$|z|>r^{\theta}$.
Then we have
$$F_{t,s_2,\theta}^\mu (x,r)\leq F_{t,s_1,\theta}^\mu (x,r),$$
and therefore $\dim_P^{s_1}\mu \leq \dim_P^{s_2}\mu$.

Similarly, only the last term of the kernel depends on $\theta$, and it
is nonincreasing in $\theta$; hence $\theta\mapsto\dim_{P,\theta}^{s}\mu$
is nondecreasing on $(0,1]$, and in particular the limit in part (2) of
Theorem~\ref{Lemma:cotainferior} exists.
\end{remark}
\begin{remark}\label{rem:theta_one}
For $\theta=1$ and $0\leq t\leq s$ one has
$\phi_{r,1}^{t,s}(z)=\min\{1,\ r^{s}|z|^{-s}\}$. Hence $F_{t,s,1}^{\mu}=F_{s}^{\mu}$ and
$$\dim_{P,1}^{s}\mu=\dim_{P}^{s}\mu\qquad\text{for every }s>0.$$
\end{remark}

Our first main result describes the limiting behavior of the
intermediate profiles as $\theta\to 0$. Besides its intrinsic
interest, it is a key ingredient in the proof of Theorem~\ref{thm:main}.

\begin{theorem}\label{Lemma:cotainferior}
    Let $\mu$ be a Borel probability measure with compact support in
    $\mathbb{R}^n$ and let $s\in(0,n]$. Then
    \begin{enumerate}
        \item For every $\theta\in(0,1]$,
        $$(1-\theta)\min\{s,\dim_{P,\theta}\mu\}\leq \dim_P^s\mu\leq s.$$
        \item $\displaystyle\lim_{\theta\to 0}\dim_{P,\theta}^s\mu=\min\left\{s,D(\mu)\right\}$.
    \end{enumerate}
\end{theorem}

Our second main result is the threshold characterization of full packing
dimension for images under random $\alpha$-H\"{o}lder maps.

\begin{theorem}\label{thm:main}
    Let $(\Omega,\mathcal{F},\mathbb{P})$ be a probability space, $\mu$ be a Borel probability measure with compact support on
    $\mathbb{R}^n$, $m\in\mathbb{N}$ and $0<\alpha\leq 1$. If
    $\{f_\omega:\operatorname{supp}(\mu)\to\mathbb{R}^m,\,\omega\in\Omega\}$
    is a family of $\alpha$-Hölder continuous, $\sigma(\mathcal{F}\otimes
    \mathcal{B}(\operatorname{supp}\mu))$-measurable functions such that
    for every $\theta\in(0,1]$ there exists $c(\theta)>0$ with
    $$\mathbb{P}\bigl(\{\omega\,:\,|f_\omega(x)-f_\omega(y)|\leq r\}\bigr)\leq
    c(\theta)\,\phi_{r^{1/\alpha},\theta}^{m\alpha,m\alpha}(x-y)$$
    for all $x,y\in\operatorname{supp}(\mu)$ and $r>0$, then
    
    $$\dim_P\mu_{f_\omega}=m \text{ almost surely}\quad\Longleftrightarrow\quad
    m\alpha\leq D(\mu).$$
\end{theorem}

The proofs of Theorems~\ref{Lemma:cotainferior} and~\ref{thm:main} are
deferred to Section~\ref{sec:proofs}. In the remainder of this section
we establish the two estimates on image measures on which they rely. Both
adapt arguments of Burrell \cite{Burrell}, where the analogous bounds
are obtained for the intermediate dimensions of random images of sets.

\begin{lemma}\label{Th1}
    Let $\mu$ be a Borel probability measure on $\mathbb{R}^n$ with compact
    support $E:=\operatorname{supp}(\mu)$, let $\theta\in(0,1]$,
    $m\in\mathbb{N}$, and $f:E\to \mathbb{R}^m$. If there exist $c>0$ and
    $0<\alpha \leq 1$ such that
    \begin{equation}\label{Holder_Condition}
        |f(x)-f(y)|\leq c\, |x-y|^{\alpha}
    \end{equation}
    for all $x,y\in E$, then
    $$\dim_{P,\theta} \mu_f \leq \frac{\dim_{P,\theta}^{m\alpha}\mu}{\alpha}.$$
\end{lemma}
\begin{proof}
    By definition and inequality \eqref{Holder_Condition}, for all $x,y\in E$,
    \begin{align*}
        \phi_{r^{1/\alpha},\theta}^{s\alpha,m\alpha}(x-y)
        &= \min \left\{1,\, \frac{r^s}{|x-y|^{s\alpha}},\, \frac{r^{\theta (m-s)+s}}{|x-y|^{m\alpha}} \right\}\\
        &\leq \min \left\{1,\, \frac{c^s\,r^s}{|f(x)-f(y)|^s},\, \frac{c^m\,r^{\theta(m-s)+s}}{|f(x)-f(y)|^m} \right\}\\
        &\leq C\, \phi_{r,\theta}^{s,m}(f(x)-f(y)),
    \end{align*}
    for some constant $C>0$.

Let $s$ be such that $s\alpha>\dim_{P,\theta}^{m\alpha}\mu$. Then there
exists a set $A\subset E$ with $\mu(A)>0$ such that, for every $x\in A$,
there exist $\epsilon(x)>0$ and $r_0(x)>0$ with

\begin{align*}
        \epsilon(x)\,r^{s} &< \int\phi_{r^{1/\alpha},\theta}^{s\alpha,m\alpha}(x-y)\,d\mu(y)\\
        &\leq C\int\phi_{r,\theta}^{s,m}(f(x)-f(y))\, d\mu(y)\\
        &= C\int\phi_{r,\theta}^{s,m}(f(x)-w)\, d\mu_f(w)\\
        &=C\,F_{s,m,\theta}^{\mu_f}\bigl(f(x),r\bigr).
    \end{align*}
    
for all $0<r<r_0(x)$, by the kernel comparison above and the definition
of the image measure. Hence
$\liminf_{r\to0}r^{-s}F_{s,m,\theta}^{\mu_f}(f(x),r)>0$ for every
$x\in A$, and since
$A\subseteq f^{-1}\bigl(\{y:\liminf_{r\to0}r^{-s}F_{s,m,\theta}^{\mu_f}(y,r)>0\}\bigr)$,
this set has positive $\mu_f$-measure. Thus the defining condition of
$\dim_{P,\theta}^{m}\mu_f$ fails at the exponent $s$, and since the set
of admissible exponents is downward closed
(see \cite[Remark 4.4]{Angelini_Packing}),
$\dim_{P,\theta}^{m}\mu_f\leq s$. As $\mu_f$ is a measure on
$\mathbb{R}^m$, $\dim_{P,\theta}\mu_f=\dim_{P,\theta}^{m}\mu_f\leq s$,
and since $s$ with $s\alpha>\dim_{P,\theta}^{m\alpha}\mu$ was arbitrary,
the conclusion follows.
\end{proof}

To handle random maps, and in particular fractional Brownian motion, we will
use a truncated version of the kernel $\phi_{r,\theta}^{s,t}$ introduced by
Burrell in \cite{Burrell}, together with a key inequality proved in the same
article.

\begin{definition}\label{defkernel2}
    Let $\theta\in(0,1]$, $s\geq 0$ and $r\in (0,1)$. Define
    $\widetilde{\phi}_{r,\theta}^s:\mathbb{R}^n\to\mathbb{R}$ by
    $$\widetilde{\phi}_{r,\theta}^s(x)=\begin{cases}
        1, & 0\leq |x|<r,\\[1mm]
        \left(\dfrac{r}{|x|}\right)^s, & r\leq |x|\leq r^\theta,\\[1mm]
        0, & |x| > r^\theta.
    \end{cases}$$
\end{definition}

The following lemma is proved in \cite{Burrell}.
\begin{lemma}\label{DesigualdadBurrell}
    Let $(\Omega,\mathcal{F},\mathbb{P})$ be a probability space, $E\subset\mathbb{R}^n$
    a compact set, $\theta\in(0,1]$, $\gamma>0$, $m\in\mathbb{N}$ and
    $s\in[0,m)$. If $\{f_\omega: E\to \mathbb{R}^m,\, \omega\in\Omega\}$ is a
    family of continuous functions that are measurable with respect
    to $\sigma(\mathcal{F}\otimes\mathcal{B}(E))$, and there exists $c>0$
    such that
    $$\mathbb{P}\bigl(\{\omega\, :\, |f_\omega (x)-f_\omega (y)|\leq r\}\bigr)
    \leq c\, \phi_{r^\gamma,\theta}^{m/\gamma,m/\gamma}(x-y)$$
    for all $x,y\in E$ and $r>0$, then there exists $C_{s,m}>0$ such that
    $$\int \widetilde{\phi}_{r,\theta}^s(f_\omega(x)-f_\omega(y))\, d\mathbb{P}(\omega)
    \leq C_{s,m}\, \phi_{r^\gamma,\theta}^{s/\gamma,m/\gamma}(x-y).$$
\end{lemma}

Combining Definition~\ref{defkernel2} and Lemma~\ref{DesigualdadBurrell}, we now
establish a lower bound for the packing intermediate profile of the image
measure under a random family of continuous maps satisfying certain
conditions.

\begin{lemma}\label{Th2}
    Let $(\Omega,\mathcal{F},\mathbb{P})$ be a probability space, $\mu$ a
    Borel probability measure on $\mathbb{R}^n$ with compact support,
    $\theta\in(0,1]$, $m\in\mathbb{N}$ and $\gamma\geq 1$. If
    $\{f_\omega:\operatorname{supp}(\mu)\to\mathbb{R}^m,\,\omega\in\Omega\}$
    is a family of continuous, $\sigma(\mathcal{F}\otimes\mathcal{B}(\operatorname{supp}\mu))$-measurable
    functions such that there exists $c>0$ with
    $$\mathbb{P}\bigl(\{\omega\,:\,|f_\omega(x)-f_\omega(y)|\leq r\}\bigr)\leq c\,
    \phi_{r^\gamma,\theta}^{m/\gamma,m/\gamma}(x-y)$$
    for all $x,y\in\operatorname{supp}(\mu)$ and $r>0$, then
    $$\dim_{P,\theta}\mu_{f_\omega}\geq \gamma\,\dim_{P,\theta}^{m/\gamma}\mu$$
    for $\mathbb{P}$-a.e.\ $\omega\in\Omega$.
\end{lemma}

\begin{proof}
Let $0<s/\gamma<\dim_{P,\theta}^{m/\gamma}\mu$, with $s<m$. By definition of
$\dim_{P,\theta}^{m/\gamma}\mu$, for $\mu$-a.e.\ $x$ and every $\epsilon>0$
there exists a sequence $\{\rho_k\}$ with $\rho_k\to 0$ such that
$$\int\phi_{\rho_k,\theta}^{s/\gamma,m/\gamma}(x-y)\,d\mu(y)<\epsilon\,\rho_k^{s/\gamma}.$$
Setting $r_k:=\rho_k^{1/\gamma}$, this becomes
$$\int\phi_{r_k^\gamma,\theta}^{s/\gamma,m/\gamma}(x-y)\,d\mu(y)<\epsilon\, r_k^{s}.$$
Fix $x$ in the full-measure set above. Integrating the inequality of
Lemma~\ref{DesigualdadBurrell} against $\mu$ and using Fubini's theorem,
for every $r\in(0,1)$,
\begin{align*}
&\int\Bigl[r^{-s}\int\widetilde{\phi}_{r,\theta}^{s}
\bigl(f_\omega(x)-z\bigr)\,d\mu_{f_\omega}(z)\Bigr]d\mathbb{P}(\omega)\\
&=\int r^{-s}\int\widetilde{\phi}_{r,\theta}^{s}
\bigl(f_\omega(x)-f_\omega(y)\bigr)\,d\mu(y)\,d\mathbb{P}(\omega)\\
&\leq C_{s,m}\,r^{-s}\int
\phi_{r^{\gamma},\theta}^{s/\gamma,m/\gamma}(x-y)\,d\mu(y).
\end{align*}
Along the sequence $\{r_k\}$ the right-hand
side is at most $C_{s,m}\,\epsilon$; hence, by Fatou's lemma,
$$\int\liminf_{r\to0}\Bigl[r^{-s}\int\widetilde{\phi}_{r,\theta}^{s}
\bigl(f_\omega(x)-z\bigr)\,d\mu_{f_\omega}(z)\Bigr]d\mathbb{P}(\omega)
\leq C_{s,m}\,\epsilon.$$
Note that the left-hand side no longer involves the sequence $\{r_k\}$,
which depended on $x$. Since $\epsilon>0$ was arbitrary, by using
Fubini's theorem we have,
for $\mathbb{P}$-a.e.\ $\omega\in\Omega$,
\begin{equation}\label{eq:fatou}
\liminf_{r\to0}\,r^{-s}\int\widetilde{\phi}_{r,\theta}^{s}
\bigl(f_\omega(x)-z\bigr)\,d\mu_{f_\omega}(z)=0
\qquad\text{for }\mu\text{-a.e.\ }x.
\end{equation}
Fix such an $\omega$, let $x$ be a point where \eqref{eq:fatou} holds,
and choose a sequence $r_k\to0$
realizing the $\liminf$. From the middle branch of
$\widetilde{\phi}_{r_k,\theta}^{s}$, for every
$\delta\in[r_k,r_k^{\theta}]$,
$$\Bigl(\frac{r_k}{\delta}\Bigr)^{s}\,
\mathbbm{1}_{B(f_\omega(x),\delta)}(z)\leq \widetilde{\phi}_{r_k,\theta}^{s}\bigl(f_\omega(x)-z\bigr),$$
whence $$\delta^{-s}\mu_{f_\omega}\bigl(B(f_\omega(x),\delta)\bigr)\leq
r_k^{-s}\int\widetilde{\phi}_{r_k,\theta}^{s}
\bigl(f_\omega(x)-z\bigr)\,d\mu_{f_\omega}(z).$$
Hence, for every $\epsilon'>0$ there are arbitrarily small windows
$[r_k,r_k^{\theta}]$ on which
$\delta^{-s}\mu_{f_\omega}(B(f_\omega(x),\delta))<\epsilon'$ for
all $\delta\in[r_k,r_k^\theta]$. Therefore
$s\leq\dim_{P,\theta}\mu_{f_\omega}$, and letting
$s\to\gamma\,\dim_{P,\theta}^{m/\gamma}\mu$ along a countable sequence
(intersecting the corresponding almost sure events) completes the proof.
\end{proof}

Combining Lemmas~\ref{Th1} and~\ref{Th2} we obtain the exact value of
the intermediate packing intermediate dimension of images under fractional Brownian motion;
at $\theta=1$, by Remark~\ref{rem:theta_one}, this recovers Xiao's
theorem \cite[Theorem 3.1]{xiao1997packing} for compactly supported
probability measures.

\begin{corollary}\label{cor:FBM}
    Let $\theta \in (0,1]$, $m,n\in \mathbb{N}$, $0<\alpha<1$, let
    $X:\mathbb{R}^n\to\mathbb{R}^m$ be an index-$\alpha$ fractional Brownian
    motion, and let $\mu$ be a Borel probability measure with compact support
    in $\mathbb{R}^n$. Then
    $$\dim_{P,\theta}\mu_{X}=\frac{\dim_{P,\theta}^{m\alpha}\mu}{\alpha}
    \qquad \text{a.s.}$$
\end{corollary}
\begin{proof}
Let $0<\epsilon<\alpha$. As recalled in section 2, there
exists $M>0$ such that
$$|X(x)-X(y)|\leq M\,|x-y|^{\alpha-\epsilon}\qquad a.s.\text{ for all }x,y\in\operatorname{supp}(\mu).$$
Moreover, by \cite[Theorem 3.4]{Burrell}, we have
$$\mathbb{P}\bigl(|X(x)-X(y)|\leq r\bigr)\leq 2^{\max\{m,n\}}\,
\phi_{r^{\gamma},\theta}^{m/\gamma,m/\gamma}(x-y)$$
for all $x,y\in\operatorname{supp}(\mu)$, where $\gamma=1/\alpha$.

Applying Lemma~\ref{Th1} to the Hölder condition with exponent
$\alpha-\epsilon$, we obtain
$$\dim_{P,\theta}\mu_X\leq \frac{\dim_{P,\theta}^{m(\alpha-\epsilon)}\mu}{\alpha-\epsilon}
\qquad \text{a.s.}$$
Now, by using Remark ~\ref{thm:lipschitz_profile} and letting $\epsilon\to 0$ yields
$$\dim_{P,\theta}\mu_X\leq \frac{\dim_{P,\theta}^{m\alpha}\mu}{\alpha}
\qquad \text{a.s.}$$
On the other hand, applying Lemma~\ref{Th2} with $\gamma=1/\alpha$ gives
the reverse inequality
$$\dim_{P,\theta}\mu_X\geq \gamma\,\dim_{P,\theta}^{m/\gamma}\mu
=\frac{\dim_{P,\theta}^{m\alpha}\mu}{\alpha}\qquad \text{a.s.}$$
Combining both, the result follows.
\end{proof}

\section{Proofs of the main theorems}\label{sec:proofs}

We begin with the proof of Theorem~\ref{Lemma:cotainferior}, and then
turn to Theorem~\ref{thm:main}, whose two implications are established
separately in Lemmas~\ref{lem:main_direction1}
and~\ref{lem:main_direction2}.

\begin{proof}[Proof of Theorem \ref{Lemma:cotainferior}]
We first prove $(1)$. The upper bound $\dim_P^s\mu\leq s$ is immediate from
the definition of the classical packing dimension profile.

For the lower bound, we begin with the comparison inequality
\begin{equation}\label{equation111}
F_{t,s,\theta}^\mu (x,r)\leq F_{t,n,\theta}^{\mu}(x,r)+ r^{t(1-\theta)},
\end{equation}
valid for all $x\in\mathbb{R}^n$, $r\in(0,1)$ and $t\in(0,s]$.

Next, observe that for every $t\in(0,s]$,
$$F_s^{\mu}(x,r)\leq F_{t,s,\theta}^\mu (x,r),$$
which follows by direct comparison of the kernels branch by branch.
Combining this with \eqref{equation111},
\begin{equation}\label{equation222}
F_s^\mu(x,r)\leq F_{t,n,\theta}^\mu(x,r)+r^{t(1-\theta)}.
\end{equation}

Fix $\theta\in(0,1]$ and let $0<t<\min\{s,\dim_{P,\theta}\mu\}$. By the
definition of $\dim_{P,\theta}\mu=\dim_{P,\theta}^n\mu$, for every $\epsilon>0$
and $\mu$-a.e.\ $x$, there exists a sequence $r_k\to 0$ such that
$$F_{t,n,\theta}^\mu(x,r_k)<\epsilon\, r_k^t.$$
Substituting into \eqref{equation222} and using $r_k<1$ together with
$t(1-\theta)\leq t$,
$$F_s^\mu(x,r_k)<\epsilon\, r_k^t+r_k^{t(1-\theta)}\leq (1+\epsilon)\,r_k^{t(1-\theta)},$$
so that
$$\liminf_{k\to\infty}r_k^{-t(1-\theta)}F_s^\mu(x,r_k)\leq 1+\epsilon<\infty.$$

Now observe that $F_s^\mu$ does not depend on $t$; therefore, for any
$t'<t(1-\theta)$,
$$r_k^{-t'}F_s^\mu(x,r_k)=r_k^{t(1-\theta)-t'}\cdot r_k^{-t(1-\theta)}F_s^\mu(x,r_k)
\leq r_k^{t(1-\theta)-t'}(1+\epsilon)\longrightarrow 0$$
as $r_k\to 0$. Hence $t'\leq \dim_P^s\mu$.
Letting $t'\to t(1-\theta)$ and then $t\to\min\{s,\dim_{P,\theta}\mu\}$, we
conclude
$$(1-\theta)\min\{s,\dim_{P,\theta}\mu\}\leq \dim_P^s\mu,$$
completing the proof of $(1)$.

To prove part $(2)$, note that the case $s\leq D(\mu)$ is absorbed by part
$(1)$ of the theorem. Indeed, letting $\theta\to 0$ in part $(1)$ and using $s\leq
D(\mu)$ gives $\dim_P^s\mu=s$. If $s=n$ the conclusion is immediate, since
$\dim_{P,\theta}^n\mu=\dim_{P,\theta}\mu$; hence assume $s<n$ and let
$X:\mathbb{R}^n\to\mathbb{R}^n$ be an index-$\alpha$ fractional Brownian
motion with $\alpha=s/n\in(0,1)$. By Theorem~\ref{ProjPacking},
$\dim_P\mu_X=\dim_P^s\mu/\alpha=n$ almost surely. Arguing as in the proof
of Lemma~\ref{lem:main_direction1} (via
Lemma~\ref{lem:threshold_lower_bound}), this yields
$\dim_{P,\theta}\mu_X=n$, and combining with Corollary~\ref{cor:FBM},
$$n=\dim_{P,\theta}\mu_X=\frac{\dim_{P,\theta}^s\mu}{\alpha}\leq n$$
for every $\theta\in(0,1]$, which implies
$s=\lim_{\theta\to 0}\dim_{P,\theta}^s\mu=\min\{s,D(\mu)\}$.

Hence, suppose $D(\mu)<s$ and let $t<D(\mu)$. By examining the kernel
separately on the intervals $[0,r^\theta]$ and $(r^\theta,\infty)$, it is not
hard to see that
\begin{equation}\label{equation333}
    \phi_{r,\theta}^{t,s}(x)\leq \phi_{r,\theta}^{t,n}(x) + r^{t(1-\theta)}.
\end{equation}
Now, by our choice of $t$, for every $\theta\in(0,1]$, for $\mu$-a.e.\ $x$
and for every $\epsilon\in(0,1)$, there exists a sequence $\{r_k\}$ with
$r_k\to 0$ such that
$$\int\phi_{r_k,\theta}^{t,n}(x-y)\,d\mu(y)\leq \epsilon\, r_k^t.$$
Combining this with inequality \eqref{equation333} we obtain
$$\int \phi_{r_k,\theta}^{t,s}(x-y)\,d\mu(y)\leq 2\, r_k^{t(1-\theta)}.$$
Therefore
\begin{equation}\label{equation444}
\liminf_{r\to 0} r^{-t(1-\theta)}\int \phi_{r,\theta}^{t,s}(x-y)\,d\mu(y)\leq 2
\end{equation}
for every $\theta\in(0,1]$ and $\mu$-a.e.\ $x$.

We now upgrade \eqref{equation444} to the exponent $t$. Fix $\theta\in(0,1]$
and let $0<u<t$. By examining the kernel we have
\begin{equation}\label{equation555}
\phi_{r,\theta}^{u,s}(x)\leq r^{(u-t)(1-\theta)}\,\phi_{r,\theta}^{t,s}(x),
\end{equation}
as can be checked directly on each of the three branches of the kernel, and
that $\phi_{r,\theta}^{t,s}$ is nonincreasing in $\theta$, so that
\begin{equation}\label{equation666}
F_{t,s,\theta}^\mu(x,r)\leq F_{t,s,\theta'}^\mu(x,r)\qquad \text{whenever }\theta'\leq\theta.
\end{equation}
Choose $\theta'\in(0,\theta]$ with $\theta'<\theta(t-u)/t$, and let
$\{r_j\}$ be a sequence realizing \eqref{equation444} for the parameter
$\theta'$, so that $F_{t,s,\theta'}^\mu(x,r_j)\leq 3\,r_j^{t(1-\theta')}$ for
all $j$ large enough. Combining \eqref{equation555} and \eqref{equation666},
\begin{align*}
    r_j^{-u}F_{u,s,\theta}^\mu(x,r_j)&\leq r_j^{-u+(u-t)(1-\theta)}F_{t,s,\theta}^\mu(x,r_j)\\
&\leq r_j^{-u+(u-t)(1-\theta)}F_{t,s,\theta'}^\mu(x,r_j)\\
&\leq 3\, r_j^{\,\theta(t-u)-t\theta'},
\end{align*}

By the choice of $\theta'$, this exponent is strictly positive, and hence
$$\liminf_{r\to 0}r^{-u}F_{u,s,\theta}^\mu(x,r)=0\qquad \text{for }\mu\text{-a.e.\ }x.$$
Since
$0<u<t$ was arbitrary, we conclude $\dim_{P,\theta}^s\mu\geq t$ for every
$\theta\in(0,1]$. As $t<D(\mu)$ was arbitrary, this gives
$$\lim_{\theta\to 0}\dim_{P,\theta}^s\mu\geq D(\mu).$$
For the reverse inequality, monotonicity of the profile in its upper index
gives $\dim_{P,\theta}^s\mu\leq \dim_{P,\theta}^n\mu=\dim_{P,\theta}\mu$ for
$s\leq n$, whence $\lim_{\theta\to 0}\dim_{P,\theta}^s\mu\leq D(\mu)$.
Therefore
$$\lim_{\theta\to 0}\dim_{P,\theta}^s\mu=D(\mu)=\min\{s,D(\mu)\},$$
since $D(\mu)<s$ in this case, completing the proof.
\end{proof}

We now turn to Theorem~\ref{thm:main}, whose proof we split into two
lemmas, one for each implication. The forward implication is
deterministic, through the
upper bound of Lemma~\ref{Th1}. The reverse implication is a
probabilistic estimate, through the lower bound of
Lemma~\ref{Th2}.

\begin{lemma}\label{lem:main_direction1}
    Let $\mu$ be a Borel probability measure with compact support on
    $\mathbb{R}^n$, $m\in\mathbb{N}$, $0<\alpha\leq 1$, and let
    $f:\operatorname{supp}(\mu)\to\mathbb{R}^m$ be an $\alpha$-Hölder
    continuous function. If $\dim_P\mu_f=m$, then $m\alpha\leq D(\mu)$.
\end{lemma}
\begin{proof}
    Since $\operatorname{supp}(\mu_f)\subset\mathbb{R}^m$, we have
    $\dim_P\mu_f\leq \dim_A(\operatorname{supp}\mu_f)\leq m$. Hence, the
    assumption $\dim_P\mu_f=m$ forces $\dim_A(\operatorname{supp}\mu_f)=m$
    as well. Applying Lemma~\ref{lem:threshold_lower_bound},
    $$\dim_{P,\theta}\mu_f\geq \dim_A(\operatorname{supp}\mu_f)-
    \frac{\dim_A(\operatorname{supp}\mu_f)-\dim_P\mu_f}{\theta}=m,$$
    we
    obtain $\dim_{P,\theta}\mu_f=m$ for every $\theta\in(0,1]$.

    Combining this with Lemma~\ref{Th1} and the monotonicity of
    $\dim_{P,\theta}^t\mu$ in $t$, we obtain
    \begin{align*}
        m=\dim_{P,\theta}\mu_f &\leq \frac{\dim_{P,\theta}^{m\alpha}\mu}{\alpha}\\
        &\leq \frac{\dim_{P,\theta}^{n}\mu}{\alpha}\\
        &= \frac{\dim_{P,\theta}\mu}{\alpha}
    \end{align*}
    for every $\theta\in(0,1]$. Therefore
    $$m\alpha\leq \lim_{\theta\to 0}\dim_{P,\theta}\mu=D(\mu). $$
\end{proof}

\begin{lemma}\label{lem:main_direction2}
    Let $(\Omega,\mathcal{F},\mathbb{P})$ be a probability space, $\mu$ a
    Borel probability measure with compact support on $\mathbb{R}^n$,
    $m\in\mathbb{N}$ and $0<\alpha\leq 1$. Suppose that
    $\{f_\omega:\operatorname{supp}(\mu)\to\mathbb{R}^m,\,\omega\in\Omega\}$
    is a family of continuous,
    $\sigma(\mathcal{F}\otimes\mathcal{B}(\operatorname{supp}\mu))$-measurable
    functions such that for every $\theta\in(0,1]$ there exists
    $c(\theta)>0$ with
    $$\mathbb{P}\bigl(\{\omega\,:\,|f_\omega(x)-f_\omega(y)|\leq r\}\bigr)\leq
    c(\theta)\,\phi_{r^{1/\alpha},\theta}^{m\alpha,m\alpha}(x-y)$$
    for all $x,y\in\operatorname{supp}(\mu)$ and $r>0$. If in addition
    $m\alpha\leq D(\mu)$, then
    $$\dim_P\mu_{f_\omega}=m\qquad \text{almost surely.}$$
\end{lemma}
\begin{proof}
    Fix a sequence $\theta_j\to 0$. For each $j$, Lemma~\ref{Th2} applied
    with $\gamma=1/\alpha$ yields
    $$\dim_{P,\theta_j}\mu_{f_\omega}\geq \frac{\dim_{P,\theta_j}^{m\alpha}\mu}{\alpha}
    \qquad\text{almost surely.}$$
    Intersecting the corresponding almost-sure events over $j$ (a countable
    intersection, hence still of full probability), we obtain a set
    $\Omega_0\subset\Omega$ with $\mathbb{P}(\Omega_0)=1$ such that the
    above holds for every $j$ and every $\omega\in\Omega_0$.

    Fix $\omega\in\Omega_0$. Since $\dim_{P,\theta}\mu_{f_\omega}\leq
    \dim_P\mu_{f_\omega}$ for every $\theta\in(0,1]$, we have
    $$\dim_P\mu_{f_\omega}\geq \dim_{P,\theta_j}\mu_{f_\omega}\geq
    \frac{\dim_{P,\theta_j}^{m\alpha}\mu}{\alpha}.$$
    Letting $j\to\infty$ and applying Lemma~\ref{Lemma:cotainferior} together
    with the assumption $m\alpha\leq D(\mu)$,
    $$\dim_P\mu_{f_\omega}\geq \frac{1}{\alpha}\lim_{j\to\infty}
    \dim_{P,\theta_j}^{m\alpha}\mu=\frac{m\alpha}{\alpha}=m.$$
    Combined with the trivial bound $\dim_P\mu_{f_\omega}\leq m$, this gives
    $\dim_P\mu_{f_\omega}=m$ for every $\omega\in\Omega_0$, completing the
    proof.
\end{proof}

\begin{proof}[Proof of Theorem \ref{thm:main}]
    Follows directly from Lemmas \ref{lem:main_direction1} and \ref{lem:main_direction2}
\end{proof}

Since
$\theta\mapsto\dim_{P,\theta}\mu$ is nondecreasing and
$\dim_{P,1}\mu=\dim_P\mu$, one always has $D(\mu)\leq\dim_P\mu$; the
following corollaries describe the situation when this inequality is an
equality, in which the profiles admit a closed form and Xiao's theorem
reduces to the packing analogue of Kahane's formula \eqref{eq:kahane}.
Note that by Lemma \ref{lem:threshold_lower_bound}, the equality $D(\mu)=\dim_P\mu$ holds, for instance, whenever
$\dim_P\mu=\dim_A(\operatorname{supp}\mu)$.

\begin{corollary}\label{cor:full_dim_profile}
    Let $\mu$ be a Borel probability measure with compact support on
    $\mathbb{R}^n$ with $\dim_P\mu=D(\mu)$. Then, for every $s>0$,
    $$\dim_P^s\mu=\min\{s,\dim_P\mu\}.$$
\end{corollary}
\begin{proof}
 Applying
    part $(2)$ of Theorem~\ref{Lemma:cotainferior}, we have
    $$\min\{s,\dim_P \mu\}=\min\{s,D(\mu)\}\leq\dim_P^s\mu\leq \min \{ s, \dim_P \mu \}.$$
\end{proof}

\begin{corollary}\label{cor:full_dim_fbm}
    Let $\mu$ be a Borel probability measure with compact support on
    $\mathbb{R}^n$ with $\dim_P\mu=D(\mu)$, and let
    $X:\mathbb{R}^n\to\mathbb{R}^m$ be an index-$\alpha$ fractional
    Brownian motion. Then
    $$\dim_P\mu_X=\min\!\left\{m,\,\frac{\dim_P \mu}{\alpha}\right\}\qquad\text{a.s.}$$
\end{corollary}
\begin{proof}
    By Theorem~\ref{ProjPacking},
    $\dim_P\mu_X=\dim_P^{m\alpha}\mu/\alpha$ almost surely, and the result follows by applying 
    Corollary~\ref{cor:full_dim_profile}.
\end{proof}
\section{A dimension threshold for images of analytic sets}\label{sec:sets}

In this final section we prove the analogue of Theorem~\ref{thm:main} for
analytic sets. The corresponding characterization for the intermediate and
box dimensions of images of compact sets was obtained in
\cite[Corollary 3.5]{angelini2025critical}. There it is shown that, for an
index-$\alpha$ fractional Brownian motion $X:\mathbb{R}^n\to \mathbb{R}^m$,
$\dim_\theta X(E)=m$ almost surely if and only if
$\alpha m\leq \dim_{qH}E$, where
$\dim_{qH}E$ denotes the quasi-Hausdorff
dimension of $E$. Theorem~\ref{thm:main_sets} below is the
packing-dimension counterpart of that result, with the parameter $D(E)$
playing the role of $\dim_{qH}E$. The passage from measures to sets rests
on a characterization, due to Xiao, of the packing dimension of the image
of a set in terms of the packing dimensions of the images of measures
supported on it. We begin by recalling the relevant definitions and
results from \cite{xiao1997packing}.

Recall that given an analytic set $E\subset\mathbb{R}^n$, we write $\mathcal{M}_c^+(E)$
for the family of finite Borel measures on $\mathbb{R}^n$ with compact
support contained in $E$. The packing dimension of $E$ can be recovered
from the measures it supports, namely
\begin{equation}\label{eq:dimP_set_measures}
\dim_P E=\sup\{\dim_P\mu\,:\,\mu\in\mathcal{M}_c^+(E)\},
\end{equation}
and the packing dimension profile of a set is defined analogously: for
$s>0$,
$$\dim_P^s E:=\sup\{\dim_P^s\mu\,:\,\mu\in\mathcal{M}_c^+(E)\}.$$

The following two results of Xiao will be needed.

\begin{lemma}[\cite{xiao1997packing}, Lemma 4.3]\label{lem:xiao_sets}
    Let $E\subset\mathbb{R}^n$ be an analytic set. Then, for any
    continuous function $f:\mathbb{R}^n\to\mathbb{R}^m$,
    $$\dim_P f(E)=\sup\{\dim_P\mu_f\,:\,\mu\in\mathcal{M}_c^+(E)\}.$$
\end{lemma}

\begin{theorem}[\cite{xiao1997packing}, Theorem 4.1]\label{thm:xiao_sets}
    Let $X:\mathbb{R}^n\to\mathbb{R}^m$ be an index-$\alpha$ fractional
    Brownian motion. Then, for every analytic set $E\subset\mathbb{R}^n$,
    $$\dim_P X(E)=\frac{\dim_P^{\alpha m}E}{\alpha}\qquad\text{a.s.}$$
\end{theorem}

In analogy with the threshold parameter of a measure, we introduce its
counterpart for sets.

\begin{definition}\label{def:threshold_set}
    For an analytic set $E\subset\mathbb{R}^n$, define
    $$\dim_{P,\theta}E:=\sup\{\dim_{P,\theta}\mu\,:\,\mu\in\mathcal{M}_c^+(E)\}$$ and
    $$D(E):=\lim_{\theta\to 0} \dim_{P,\theta}E.$$
\end{definition}

Note that for
each $\mu$, the map $\theta\mapsto\dim_{P,\theta}\mu$ is nondecreasing,
hence so is the supremum, and then the limit as $\theta\to 0$ exists and is its infimum
over $\theta\in(0,1]$.

We are now ready to state the analogue of Theorem~\ref{thm:main} for
analytic sets.

\begin{theorem}\label{thm:main_sets}
    Let $E\subset\mathbb{R}^n$ be an analytic set and $s>0$. Then
    $$\dim_P^s E=s \quad\Longleftrightarrow\quad
    s\leq D(E).$$
\end{theorem}

\begin{proof}
Fix $\alpha\in (0,1)$, $m\in\mathbb{N}$ such that $s=\alpha m$ and let $X:\mathbb{R}^n\to\mathbb{R}^m$ be an
    index-$\alpha$ fractional Brownian motion.

Suppose first that $\dim_P^s E=s$, this is  $\dim_P X(E)=m$ almost surely. Fix $\omega$ in the
almost-sure event where $\dim_P X_\omega(E)=m$ and where $X_\omega$ is
$(\alpha-\epsilon')$-H\"{o}lder continuous on compact sets for every
$\epsilon'>0$. Fix $\theta\in(0,1]$ and $\epsilon>0$. By
Lemma~\ref{lem:xiao_sets}, there exists $\mu\in\mathcal{M}_c^+(E)$,
which we may normalize to be a probability measure, such that
$\dim_P\mu_{X_\omega}>m-\epsilon$. Since
$$m-\epsilon<\dim_P\mu_{X_\omega}\leq
\dim_A(\operatorname{supp}\mu_{X_\omega})\leq m,$$
Lemma~\ref{lem:threshold_lower_bound} yields
$$\dim_{P,\theta}\mu_{X_\omega}\geq
\dim_A(\operatorname{supp}\mu_{X_\omega})-
\frac{\dim_A(\operatorname{supp}\mu_{X_\omega})-\dim_P\mu_{X_\omega}}{\theta}
\geq m-\epsilon-\frac{\epsilon}{\theta}.$$
On the other hand, by Lemma~\ref{Th1} applied to $X_\omega$ with
H\"{o}lder exponent $\alpha-\epsilon'$, together with Remark ~\ref{thm:lipschitz_profile} and letting $\epsilon'\to 0$,
$$\frac{\dim_{P,\theta}^{m\alpha}\mu}{\alpha}\geq \dim_{P,\theta}\mu_{X_\omega}
\geq m-\epsilon-\frac{\epsilon}{\theta}.$$
Letting $\epsilon\to 0$ with $\theta$ fixed, and using the trivial bound
$\dim_{P,\theta}^{m\alpha}\mu\leq m\alpha$, we obtain
$$\sup_{\mu\in\mathcal{M}_c^+(E)}\dim_{P,\theta}^{m\alpha}\mu= m\alpha
\qquad\text{for every }\theta\in(0,1].$$
Moreover, since
$\dim_{P,\theta}^{m\alpha}\mu\leq\dim_{P,\theta}\mu$ for every $\mu$, letting $\theta\to 0$ gives $$s=m\alpha\leq D(E).$$

Conversely, suppose $s\leq D(E)$. Fix $\theta\in(0,1]$ and
$\delta>0$. Since $\theta'\mapsto\sup_\mu\dim_{P,\theta'}\mu$ is
nondecreasing and its limit as $\theta'\to 0$ is $D(E)$, we have
$\sup_\mu\dim_{P,\theta}\mu\geq D(E)\geq s$, so there exists
$\mu=\mu_{\theta,\delta}\in\mathcal{M}_c^+(E)$, normalized to be a
probability measure, with
$$\dim_{P,\theta}\mu\geq s-\delta.$$
By part $(1)$ of Theorem~\ref{Lemma:cotainferior} applied to $\mu$,
$$\dim_P^{s}\mu\geq(1-\theta)\min\{s,\dim_{P,\theta}\mu\}
\geq(1-\theta)(s-\delta).$$
Taking the supremum over $\mu\in\mathcal{M}_c^+(E)$ and recalling the
definition of the profile of a set,
$$\dim_P^{s}E\geq(1-\theta)(s-\delta).$$
Since $\theta\in(0,1]$ and $\delta>0$ were arbitrary, letting
$\theta\to 0$ and $\delta\to 0$ yields $\dim_P^{s}E\geq s$,
and the reverse inequality being trivial, completing the proof.
\end{proof}

\begin{corollary}\label{cor:final}
    Let $E$ be an analytic set, $m\in\mathbb{N}$ and $0<\alpha\leq 1$. If
    $\{f_\omega:E\to\mathbb{R}^m,\,\omega\in\Omega\}$
    is a family of $\alpha$-Hölder continuous, $\sigma(\mathcal{F}\otimes
    \mathcal{B}(E))$-measurable functions such that
    for every $\theta\in(0,1]$ there exists $c(\theta)>0$ with
    $$\mathbb{P}\bigl(\{\omega\,:\,|f_\omega(x)-f_\omega(y)|\leq r\}\bigr)\leq
    c(\theta)\,\phi_{r^{1/\alpha},\theta}^{m\alpha,m\alpha}(x-y)$$
    for all $x,y\in E$ and $r>0$, then
    
    $$\dim_P f_\omega(E)=m \text{ almost surely}\quad\Longleftrightarrow\quad
    m\alpha\leq D(E).$$
\end{corollary}

\begin{proof}
Since $f_\omega(E)\subseteq\mathbb{R}^m$, we always have
$\dim_P f_\omega(E)\leq m$; moreover $f_\omega(E)$ is analytic, being a
continuous image of an analytic set.

Assume $m\alpha\leq D(E)$. By Theorem~\ref{thm:main_sets},
$\dim_P^{m\alpha}E=m\alpha$, so there exist
$\mu_j\in\mathcal{M}_c^+(E)$ with
$\dim_P^{m\alpha}\mu_j>m\alpha-1/j$. Applying Lemma~\ref{Th2} with
$\theta=1$ and $\gamma=1/\alpha$ to each $\mu_j$, together with
Remark~\ref{rem:theta_one},
$$\dim_P\mu_{j,f_\omega}=\dim_{P,1}\mu_{j,f_\omega}
\geq\tfrac{1}{\alpha}\,\dim_{P,1}^{m\alpha}\mu_j
=\tfrac{1}{\alpha}\,\dim_P^{m\alpha}\mu_j>m-\tfrac{1}{j\alpha},$$
almost surely.

Since $\mu_{j,f_\omega}$ is supported on $f_\omega(E)$, we have
$\dim_P f_\omega(E)\geq\dim_P\mu_{j,f_\omega}$, and intersecting the
countably many almost sure events yields $$\dim_P f_\omega(E)= m,$$ almost surely.

Conversely, assume $\dim_P f_\omega(E)=m$ almost surely and fix any
$\omega$ in the corresponding event. By \eqref{eq:dimP_set_measures}
and Lemma~\ref{Th1} applied to the $\alpha$-H\"{o}lder map $f_\omega$
with $\theta=1$, together with Remark~\ref{rem:theta_one},
$$m=\dim_P f_\omega(E)
=\sup_{\mu\in\mathcal{M}_c^+(E)}\dim_P\mu_{f_\omega}
\leq\tfrac{1}{\alpha}\sup_{\mu\in\mathcal{M}_c^+(E)}\dim_P^{m\alpha}\mu
=\tfrac{1}{\alpha}\,\dim_P^{m\alpha}E,$$
so $m\alpha\leq\dim_P^{m\alpha}E\leq m\alpha$; then Theorem~\ref{thm:main_sets}
gives $m\alpha\leq D(E)$.
\end{proof}

\section*{Funding}

This work was supported by FIS-2023-02725, singular structures in the geometry of measures: decompositions, rigidity and rectifiability, CUP E53C25001800001. The author was additionally partially supported by grants PICT 2022-4875 (ANPCyT), PIP 202287/22 (CONICET), and PROICO 3-0720 ``An\'alisis Real y Funcional. Ec. Diferenciales''.

\section*{Declaration of generative AI and AI-assisted technologies in the writing process}

During the preparation of this work the author used Claude (Anthropic) in
order to improve the clarity, grammar, and English style of the
manuscript, and to help with LaTeX formatting. After
using this tool, the author reviewed and edited the content as needed and
takes full responsibility for the content of the publication.

\end{document}